\documentclass{amsart}%
\usepackage[latin9]{inputenc}
\usepackage{amstext}
\usepackage{amsthm}
\usepackage{amssymb}
\usepackage{amsfonts}
\usepackage[pdftex]{graphicx}
\usepackage{amsmath}%
\providecommand{\U}[1]{\protect\rule{.1in}{.1in}}
\newtheorem{theorem}{Theorem}
\theoremstyle{plain}

\newtheorem{corollary}{Corollary}

\newtheorem{lemma}{Lemma}

\newtheorem{proposition}{Proposition}
\newtheorem{remark}{Remark}

\numberwithin{equation}{section}
\begin{document}
\title{Functional inequalities in the framework of Banach spaces}
\author{Constantin P. Niculescu}
\address{Department of Mathematics, University of Craiova, Craiova 200585, Romania}
\email{constantin.p.niculescu@gmail.com}
\date{August 15, 2024}

\begin{abstract}
A quadrilateral inequality established by Schötz \cite{Sch} in the context of
Hilbert spaces is extended to the framework of Banach spaces. Our approach is
based on the theory of majorization and a substitute for the parallelogram law
associated with Clarkson's notion of von Neumann-Jordan constant. As a
by-product, several functional inequalities that extend classical inequalities
from linear algebra and geometry of Banach spaces are also obtained.

\end{abstract}
\maketitle

\section{Introduction}

This paper aims to illustrate the usefulness of majorization theory and
higher-order convexity theory in generalizing several classical inequalities
as functional inequalities. This includes results due to Ball, Carlen, and
Lieb \cite{BCL1994}, Hanner \cite{Han56}, Popoviciu \cite{Pop1965} and Zhang
\cite{Zhang}.

Our starting point was a recent paper by Schötz \cite{Sch}, that reveals the
connection between a class of higher-order convex functions and the following
quadruple inequality,%
\begin{equation}
\left\Vert p-x\right\Vert +\left\Vert y-z\right\Vert \leq\left\Vert
p-y\right\Vert +\left\Vert x-z\right\Vert +\left\Vert p-z\right\Vert
+\left\Vert x-y\right\Vert , \label{Sch_ineq}%
\end{equation}
that occurs in any normed vector space. To see that (\ref{Sch_ineq}) indeed
works, add side by side the following two inequalities that result from the
triangle inequality,%
\[
2\left\Vert p-x\right\Vert \leq\left\Vert p-y\right\Vert +\left\Vert
p-z\right\Vert +\left\Vert y-x\right\Vert +\left\Vert z-x\right\Vert
\]
and%
\[
2\left\Vert y-z\right\Vert \leq\left\Vert y\right\Vert +\left\Vert
z\right\Vert +\left\Vert y-x\right\Vert +\left\Vert z-x\right\Vert .
\]

When dealing with inner product spaces, we also have
\begin{align*}
\left\Vert y-z\right\Vert ^{2}-\left\Vert p-y\right\Vert ^{2}-\left\Vert
x-z\right\Vert ^{2}+\left\Vert p-x\right\Vert ^{2}  &  =2\langle
z-p,x-y\rangle\\
&  \leq2\left\Vert p-z\right\Vert \left\Vert x-y\right\Vert \\
&  \leq\left\Vert p-z\right\Vert ^{2}+\left\Vert x-y\right\Vert ^{2},
\end{align*}
so the inequality (\ref{Sch_ineq}) remains valid by squaring the norms. Schötz
was able to single out an entire class of functions $f:[0,\infty
)\rightarrow\lbrack0,\infty)$ that extend this inequality by replacing
$\left\Vert \cdot\right\Vert $ with $f\left(  \left\Vert \cdot\right\Vert
\right)  $:

\begin{theorem}
\emph{(Schötz \cite{Sch}, Theorem }$3$\emph{)}\label{thm1} Let $x,y,z,p$ be
four points in an inner product space $V$ and let $f:[0,\infty)\rightarrow
\lbrack0,\infty)$ be a nondecreasing, convex, and differentiable function such
that $f(0)=0$ and $f^{\prime}$ is concave. Then%
\begin{multline*}
f\left(  \left\Vert p-x\right\Vert \right)  +f\left(  \left\Vert
y-z\right\Vert \right) \\
\leq f\left(  \left\Vert p-y\right\Vert \right)  +f\left(  \left\Vert
x-z\right\Vert \right)  +f\left(  \left\Vert p-z\right\Vert \right)  +f\left(
\left\Vert x-y\right\Vert \right)  .
\end{multline*}

\end{theorem}

For convenience, we will denote by $\mathcal{S}$ the set of all functions
$f:[0,\infty)\rightarrow\lbrack0,\infty)$ which are nondecreasing, convex and
differentiable, and have concave derivatives and by $\mathcal{S}_{0}$ the set
of functions in $\mathcal{S}$ which vanishes at the origin. Some few examples
of functions belonging to $\mathcal{S}_{0}$ are%
\begin{gather*}
\left(  1+\alpha x^{2}\right)  ^{1/2}-1\text{\quad}(\text{for }\alpha
>0),\text{ }\\
x^{\alpha}\text{\quad(for }\alpha\in\lbrack1,2]),\quad x\log\left(
x+1\right)  ,\text{ and }\log(\cosh x).
\end{gather*}

The aim of the present paper is to gain more insight into this matter by
noticing that the functions considered by Schötz belong to the subject of
higher order convexity (briefly summarized in Section 2). Indeed, the set
$\mathcal{S}$ coincides with the set of all functions $f:[0,\infty
)\rightarrow\mathbb{R}$ which are nondecreasing, convex and $3$-concave in the
sense of E. Hopf \cite{Hopf} and T. Popoviciu \cite{Pop34}, \cite{Pop44}. See
Theorem \ref{thmH_P}, Section 2. Combining this fact with Popoviciu's
approximation theorem (Theorem \ref{thm_Pop35} in Section 2) we may reduce the
reasoning with functions in $\mathcal{S}$ to the case where they are also
three times continuously differentiable.

In Section 3 we establish the "truncated form" of the two basic results of
majorization theory, the Hardy-Littlewood-Pólya theorem and the Tomi\'{c}-Weyl
theorem on weak majorization. Their main advantage is the possibility to
handle pairs of finite sequences of numbers, of different sizes. See Theorem
\ref{thm_convex_trunc_maj} and Theorem \ref{thm_concave_trunc_maj}. Unlike
Sherman's theorem of majorization (see \cite{NP2024}) our results do not make
explicit use of the stochastic matrices.

The applications proved in Section 4 represent functional generalization of
the determinantal inequalities of Zhang \cite{Zhang} and Popoviciu, of the
optimal $2$-uniform convexity inequality for $L^{p}$ spaces (with $p\in(1,2])$
and of Hanner's inequality. See respectively Theorem \ref{thm_Zhang}, Theorem
\ref{thm-extension of 2-unif cv}, Theorem \ref{thm_Hanner p<=2} and
\ref{thm_Hanner p>=2}. The $2$-uniform convexity functional inequality for
$p\in\lbrack2,\infty)$ makes the objective of Theorem
\ref{thm-extension of 2-unif cv p>=2} in Section 6.

The Section 5 is devoted to a quick presentation of a basic ingredient
necessary in our extension of Theorem 1 to the general context of Banach
spaces: the von Neumann-Jordan constant. According to Clarkson \cite{Cl1937},
the \emph{von Neumann-Jordan constant} $C_{NJ}(X)$ of a Banach space $X$ is
defined by the formula
\[
C_{NJ}(X)=\sup\left\{  \frac{\left\Vert u+v\right\Vert ^{2}+\left\Vert
u-v\right\Vert ^{2}}{2\left\Vert u\right\Vert ^{2}+2\left\Vert v\right\Vert
^{2}}:u,v\in X\text{ and }\left\Vert u\right\Vert +\left\Vert v\right\Vert
\neq0\right\}  .
\]
He noticed that $C_{NJ}(X)\in\lbrack1,2]$ and the equality $C_{NJ}(X)=1$
characterizes inner product spaces.

The \emph{von Neumann-Jordan constant }makes possible the following substitute
of the parallelogram law in any Banach space $X:$%
\[
\left\Vert x+y\right\Vert ^{2}+\left\Vert x-y\right\Vert ^{2}\leq
2C_{NJ}(X)\left(  \left\Vert x\right\Vert ^{2}+\left\Vert y\right\Vert
^{2}\right)  ,\text{\quad for all }x,y\in X.
\]

A functional inequality generalizing this fact makes the objective of Theorem
\ref{thm_functional-Par-law}. This theorem is used in Section 6 to prove the
following generalization of Theorem 1:

\begin{theorem}
\label{thm_main}Let $y,z,q,r$ be four points in the Banach space $X$ and let
$f$ be a nondecreasing, convex and $3$-concave function defined on
$[0,\infty)$ such that $f(0)=0.$ Then%
\begin{multline*}
f\left(  \left\Vert y-q\right\Vert \right)  +f\left(  \left\Vert
z-r\right\Vert \right) \\
\leq\frac{N(X)}{2}\left\{  f\left(  \left\Vert y-z\right\Vert \right)
+f\left(  \left\Vert r-q\right\Vert \right)  \right.  \left.  +f\left(
\left\Vert z-q\right\Vert \right)  +f\left(  \left\Vert y-r\right\Vert
\right)  \right\}  .
\end{multline*}

Here $N(X)=2C_{NJ}(X)$ if $2C_{NJ}(X)$ is an integer and $N(X)=4$ otherwise.
\end{theorem}

Since $N(X)=2$ if $X$ is an inner product space, in this case Theorem
\ref{thm_main} reduces to Theorem \ref{thm1}.

\section{Some basic facts concerning the $n$-convex functions $(n\leq3)$}

Higher order convexity was introduced by Hopf \cite{Hopf} and Popoviciu
\cite{Pop34}, \cite{Pop44}, who defined it in terms of divided differences.
Assuming that $f$ is a real-valued function defined on an interval $I,$ its
divided differences of order $0,1,\ldots,n$ associated to a family
$x_{0},x_{1},\ldots,x_{n}$ of $n+1$ distinct points are respectively defined
by the formulas:
\begin{align*}
\lbrack x_{0};f]  &  =f(x_{0})\\
\lbrack x_{0},x_{1};f]  &  =\frac{f(x_{1})-f(x_{0})}{x_{1}-x_{0}}\\
&  ...\\
\lbrack x_{0},x_{1},...,x_{n};f]  &  =\frac{[x_{1},x_{2},...,x_{n}%
;f]-[x_{0},x_{1},...,x_{n-1};f]}{x_{n}-x_{0}}\\
&  =%
{\displaystyle\sum\limits_{j=0}^{n}}
\frac{f(x_{j})}{\prod\nolimits_{k\neq j}\left(  x_{j}-x_{k}\right)  }.
\end{align*}

Notice that all these divided differences are invariant under the permutation
of points $x_{0},x_{1},...,x_{n}.$ As a consequence, we may always assume that
$x_{0}<x_{1}<\cdots<x_{n}.$

A function $f$ is called $n$-\emph{convex }(respectively\emph{\ }%
$n$-\emph{concave}) if all divided differences $[x_{0},x_{1},\ldots,x_{n};f]$
are nonnegative (respectively nonpositive). In particular,

\begin{itemize}
\item the convex functions of order 0 are precisely the nonnegative functions;

\item the convex functions of order 1 are the nondecreasing functions;

\item the convex functions of order 2 are nothing but the usual convex
functions since in this case for all $x_{0}<x_{1}<x_{2}$ in $I,$%
\[
\lbrack x_{0},x_{1},x_{2};f]=\frac{\frac{f(x_{0})-f(x_{1})}{x_{0}-x_{1}}%
-\frac{f(x_{1})-f(x_{2})}{x_{1}-x_{2}}}{x_{0}-x_{2}}\geq0,
\]
that is,
\[
\left(  x_{2}-x_{0}\right)  f(x_{1})\leq\left(  x_{2}-x_{1}\right)
f(x_{0})+\left(  x_{1}-x_{0}\right)  f(x_{2}).
\]

\end{itemize}

The description of 3-convex functions (as well as of the higher order convex
functions) in terms of divided differences is rather intricate. For example, a
function $f$ is $3$-\emph{convex} if for every quadruple $x_{0}<x_{1}%
<x_{2}<x_{3}$ of elements we have%
\begin{multline*}
\lbrack x_{0},x_{1},x_{2},x_{3};f]=\frac{f(x_{0})}{(x_{0}-x_{1})(x_{0}%
-x_{2})(x_{0}-x_{3})}-\frac{f(x_{1})}{(x_{0}-x_{1})(x_{1}-x_{2})(x_{1}-x_{3}%
)}\\
+\frac{f(x_{2})}{(x_{0}-x_{2})(x_{1}-x_{2})(x_{2}-x_{3})}-\frac{f(x_{3}%
)}{(x_{0}-x_{3})(x_{1}-x_{3})(x_{2}-x_{3})}\geq0,
\end{multline*}
equivalently,%
\begin{multline*}
(x_{2}-x_{0})(x_{3}-x_{0})(x_{3}-x_{2})f(x_{1})+(x_{1}-x_{0})(x_{2}%
-x_{0})(x_{2}-x_{1})f(x_{3})\\
\geq(x_{2}-x_{1})(x_{3}-x_{1})(x_{3}-x_{2})f(x_{0})+(x_{1}-x_{0})(x_{3}%
-x_{0})(x_{3}-x_{1})f(x_{2}).
\end{multline*}
When the points $x_{0},x_{1},x_{2},x_{3}$ are equidistant, that is, when
$x_{1}=x_{0}+h,$ $x_{2}=x_{0}+2h,$ $x_{3}=x_{0}+3h$ for some $h>0,$ the last
inequality becomes%
\[
f(x_{0}+3h)-3f(x_{0}+2h)+3f(x_{0}+h)-f(x_{0})\geq0,
\]
equivalently,%
\begin{equation}
f(x_{0})+3f\left(  \frac{x_{0}+2x_{3}}{3}\right)  \leq3f\left(  \frac
{2x_{0}+x_{3}}{3}\right)  +f(x_{3}). \label{eq3convJ}%
\end{equation}

Fortunately, some others, more convenient, approaches are available.

If $f$ is $3$-times differentiable, then a repeated application of the mean
value theorem yields the existence of a point $\xi\in\left(  \min_{k}%
x_{k},\max_{k}x_{k}\right)  $ such that
\[
\lbrack x_{0},x_{1},x_{2},x_{3};f]=\frac{f^{(3)}(\xi)}{6}.
\]

As a consequence, one obtains the sufficiency part of the following practical
criterion of $3$-convexity.

\begin{lemma}
\label{lemHopf}Suppose that $f$ is a continuous function defined on an
interval $I$ which is $3$-times differentiable on the interior of $I.$ Then
$f$ is $3$-convex if and only if its derivative of third order is nonnegative.
\end{lemma}

The necessity part is also immediate by using the standard formulas for
derivatives via iterated differences,%
\[
f^{(3)}(x)=\lim_{h\rightarrow0+}\frac{f(x_{0}+3h)-3f(x_{0}+2h)+3f(x_{0}%
+h)-f(x_{0})}{h^{3}}.
\]

According to Lemma \ref{lemHopf}, the following functions are $3$-convex
functions on $\mathbb{R}_{+}$:
\begin{gather*}
x^{\alpha}\text{ }(\text{for }\alpha\in(0,1]\cup\lbrack2,\infty)),\text{\quad
}x/(1+x),\text{ }\\
\log(1+x),~-x\log x,~\sinh,~\cosh,~\text{and }-\log\left(  \Gamma(x)\right)  .
\end{gather*}

Notice that the polynomials of degree $\leq2$ are both $3$-convex and
3-concave functions on the whole real line.

\begin{remark}
\emph{(}Permanence proprieties\emph{)}\label{rem_permanence prop}The
continuous $n$-convex functions defined on an interval $I$ constitute a convex
cone in the vector space $C(I),$ of all continuous functions on $I.$

Every continuous function which is $n$-convex on the interior of $I$ is
$n$-convex on the whole interval.

The limit of a pointwise convergent sequence of $n$-convex functions is also
an $n$-convex function.
\end{remark}

The following characterization of higher order convexity is due to Hopf
(\cite{Hopf}, p. 24) and Popoviciu (\cite{Pop34}, p. 48):

\begin{theorem}
\label{thmH_P}Suppose that $f$ is a continuous function defined on an interval
$I.$ Then $f$ is $3$-convex if and only if it is differentiable on the
interior of $I$ and $f^{\prime}$ is a convex function.
\end{theorem}

\begin{corollary}
\label{corH_P_S}Every function $f\in\mathcal{S}$ is $3$-concave and every
nondecreasing, convex and $3$-concave function $f:[0,\infty)\rightarrow
\lbrack0,\infty)$ belongs to $\mathcal{S}$.
\end{corollary}

\begin{proof}
It suffices to show that every nondecreasing, convex and $3$-concave function
$f:[0,\infty)\rightarrow\lbrack0,\infty)$ is continuously differentiable at
the origin. For this, notice first that%
\[
0<x<y\text{\quad implies\quad}0\leq\frac{f(x)-f(0)}{x}\leq\frac{f(y)-f(0)}%
{y},
\]
since~$f$ is nondecreasing and convex. As a consequence,%
\[
\lim_{x\rightarrow0+}\frac{f(x)-f(0)}{x}=\inf_{x>0}\frac{f(x)-f(0)}{x}\geq0,
\]
which assures the differentiability at the origin (and thus everywhere,
according to Theorem \ref{thmH_P}). Since $f$ is convex, its derivative
$f^{\prime}$ is nondecreasing. Therefore,%
\begin{align*}
\lim_{x\rightarrow0+}f^{\prime}(x)  &  =\inf_{x>0}f^{\prime}(x)=\inf
_{x,h>0}\frac{f(x+h)-f(x)}{h}\\
&  =\underset{h\rightarrow0+}{\lim\inf}\frac{f(h)-f(0)}{h}=f^{\prime}(0),
\end{align*}
which means that $f^{\prime}$ is continuous at the origin.
\end{proof}

An important source of nonnegative, nondecreasing, convex and $3$-concave
functions on a compact interval $[0,A]$ is that of completely monotone
functions. Recall that a function $f:[0,\infty)\rightarrow\lbrack0,\infty)$ is
\emph{completely monotone }if it is continuous on $[0,\infty)$, indefinitely
differentiable on $(0,\infty)$ and
\[
(-1)^{n}f^{(n)}(x)\geq0\text{\quad for all }x>0\text{ and }n\geq0.
\]

Some simple examples are $e^{-x},$ $1/(1+x),~$and $\left(  1/x\right)
\log(1+x).$ Fore more details, see the monograph of Schilling, Song and
Vondra\v{c}ek \cite{SSV}. Every completely monotone function $f$ is
nonnegative, nonincreasing, convex and $3$-concave, but adding to it linear
functions $\alpha x$ with $\alpha\geq-\inf_{x\in\lbrack0,A]}f^{\prime}(x)$ one
obtains nonnegative, nondecreasing, convex and $3$-concave functions on a
given compact interval $[0,A].$

Popoviciu has characterized the property of $n$-convexity in terms of higher
order differences.

The \emph{difference operator} $\Delta_{h}$ $($of step size $h\geq0)$
associates to each function $f$ defined on an interval $I$ the function
$\Delta_{h}f$ defined by%
\[
\left(  \Delta_{h}f\right)  (x)=f(x+h)-f(x),
\]
for all $x$ such that the right-hand side formula makes sense. Notice that no
restrictions are necessary if $I=\mathbb{R}^{+}$ or $I=\mathbb{R}.$ The
difference operators are linear and commute to each other,%
\[
\Delta_{h_{1}}\Delta_{h_{2}}=\Delta_{h_{2}}\Delta_{h_{1}}.
\]

They also verify the following property of invariance under translation:%
\[
\Delta_{h}\left(  f\circ T_{a}\right)  =\left(  \Delta_{h}f\right)  \circ
T_{a},
\]
where $T_{a}$ is the translation defined by the formula $T_{a}(x)=x+a.$

The higher order iterated differences can be introduced via the formulas:%
\begin{align*}
\left(  \Delta_{h}\right)  ^{0}f(x)  &  =f(x)\\
\left(  \Delta_{h}\right)  ^{n}f(x)  &  =\underset{n~\text{times}}%
{\underbrace{\Delta_{h}\cdots\Delta_{h}}}f(x)\\
&  =%
{\displaystyle\sum\nolimits_{k=0}^{n}}
\left(  -1\right)  ^{n-k}\binom{n}{k}f\left(  x+kh\right)  \text{\quad for
}n\geq1.
\end{align*}
Their connection with the higher order divided differences is given by%
\[
\left(  \Delta_{h}\right)  ^{n}f(x)=h^{n}[x,x+h,...,x+nh;f]
\]
and this applies to every function $f$ defined on an interval $I$ of the form
$[0,A]$ or $[0,\infty),$ all points $x\in I$ and all steps $h>0$ such that
$x+nh\in I.$

Clearly, if $f$ is an $n$-convex function $(n\geq1)$ defined on an interval
$I$ of the form $[0,A]$ or $[0,\infty),$ then%
\begin{equation}
\left(  \Delta_{h}\right)  ^{n}f(x)\geq0 \label{eq_equal differences}%
\end{equation}
for all $x\in I$ and all $h>0$ such that $x+nh\in I.$ As was noticed by
Popoviciu \cite{Pop44} (at the beginning of Section 24, p. 49)) this property
characterizes the $n$-convex functions under the presence of continuity. See
also \cite{MN2023} and \cite{NS2023} (as well as the references therein).

The inequality (\ref{eq_equal differences}) together with Bernstein's variant
of the Weierstrass approximation theorem (see \cite{CN2014}, Theorem 8.8.1, p.
256) yields the following shape preserving approximation result.

\begin{theorem}
$($\emph{Popoviciu's approximation theorem} \emph{\cite{Pop35}}$)$%
\label{thm_Pop35} If a continuous function $f:[0,1]\rightarrow\mathbb{R}$ is
$k$-convex, then so are the Bernstein polynomials associated to it,
\[
B_{n}(f)(x)=\sum_{i=0}^{n}\binom{n}{i}x^{i}(1-x)^{n-i}f\left(  \frac{i}%
{n}\right)  .
\]
Moreover, by the well-known property of simultaneous uniform approximation of
a function and its derivatives by the Bernstein polynomials and their
derivatives, it follows that $B_{n}(f)$ and any derivative \emph{(}of any
order\emph{)} of it converge uniformly to $f$ and to its derivatives, correspondingly.
\end{theorem}

Using a change of variable, one can easily see that the approximation theorem
extends to functions defined on compact intervals $[a,b]$ with $a<b.$

\begin{proof}
Using mathematical induction one can easily show that the derivatives of
Bernstein's polynomials verify the formula%
\begin{multline*}
B_{n}^{(k)}(f)(x)\\
=n(n-1)\cdots(n-k+1)\sum\nolimits_{i=0}^{n-k}\underset{k~\text{times}%
}{\underbrace{\Delta_{1/n}\cdots\Delta_{1/n}}}f(j/n)\binom{n-k}{i}%
x^{i}(1-x)^{n-k-i}.
\end{multline*}
The proof ends by taking into account the formula (\ref{eq_equal differences})
and Lemma \ref{lemHopf}.
\end{proof}

When combined with Remark \ref{rem_permanence prop}, Theorem \ref{thm_Pop35}
implies that any result valid for the smooth $n$-convex functions also works
for all $n$-convex continuous functions.

\begin{corollary}
\label{cor_comp}If $f:\mathbb{R}_{+}\rightarrow\lbrack0,\infty)$ is a
continuous $3$-convex function which is also nondecreasing and concave, then
the same properties hold for $f^{\alpha}$ if $\alpha\in(0,1].$
\end{corollary}

\begin{proof}
According to Theorem \ref{thm_Pop35}, we may reduce the proof to the case
where the involved function is of class $C^{3}$, in which case the conclusion
follows from Lemma \ref{lemHopf}.
\end{proof}

For a second application of Theorem \ref{thm_Pop35} we need the following well
known fact on concave functions.

\begin{lemma}
\label{lemf/x}If $f:[0,\infty)\rightarrow\mathbb{R}$ is a concave function and
$f(0)\geq0,$ then the function $f(x)/x$ is nonincreasing on \thinspace
$(0,\infty).$
\end{lemma}

We are now in a position to state a rather general result concerning the
composition of functions with opposite properties of convexity. It extends
Lemma 27 in \cite{Sch}.

\begin{theorem}
\label{thm-comp}Suppose that $f:[0,\infty)\rightarrow\mathbb{R}$ is a
nondecreasing, continuous and $3$-concave function. Then the function
$g(x)=f(x^{\alpha})$ is nondecreasing and concave for every $\alpha
\in(0,1/2].$
\end{theorem}

\begin{proof}
Combining Popoviciu's approximation theorem with Remark
\ref{rem_permanence prop}, we can reduce ourselves to the case where $f$ is of
class $C^{2}.$ Since $f$ is nondecreasing, it follows that $f^{\prime}\geq0.$
According to Theorem \ref{thmH_P}, $f^{\prime}$ is a concave function, so by
Lemma \ref{lemf/x} it results that $f^{\prime}(x)/x$ is nonincreasing on
$(0,\infty).$ The fact that $g$ is nondecreasing is clear. To prove that $g$
is also concave it suffices to show that its derivative is nonincreasing.
Indeed, $g^{\prime}$ can be represented as the product of two nonincreasing
nonnegative functions,%
\[
g^{\prime}(x)=\alpha x^{\alpha-1}f^{\prime}\left(  x^{\alpha}\right)  =\alpha
x^{2\alpha-1}\cdot\frac{f^{\prime}\left(  x^{\alpha}\right)  }{x^{\alpha}},
\]
and the proof is done.
\end{proof}

Theorem \ref{thm-comp} is not valid for $\alpha\in(1/2,\infty),$ a
counterexample being provided by the function $f(x)=x^{2}$, $x\geq0.$

\begin{remark}
\label{rem_cv of f (sqrt x)}

\emph{(a)} The argument of Theorem \emph{\ref{thm-comp}} also shows that
$f\left(  x^{\alpha}\right)  $ is a nonincreasing convex function if
$\alpha\in(0,1/2]$ and $f$ is a function of the same nature$;$

\emph{(b)} $f\left(  x^{\alpha}\right)  $ is a convex function provided that
$\alpha\in(0,1/2]$ and $f$ is a differentiable, convex and $3$-convex function
such that $f^{\prime}(0)\leq0$. Indeed, proceeding as in the proof of Theorem
\emph{\ref{thm-comp}} one can assume that $f$ is of class $C^{2}.$ According
to Theorem \emph{\ref{thmH_P}}, the condition of $3$-convexity implies that
$f$ has a convex derivative on $(0,\infty)$, so taking into account Lemma
\emph{\ref{lemf/x}}, the function $f^{\prime}(x)/x$ is nondecreasing,
Therefore the derivative of $f^{\prime}(x)/x$ is nonnegative, a fact that
assures that the second derivative of the function $f\left(  x^{\alpha
}\right)  $ is also nonnegative.
\end{remark}

Some few examples of differentiable and $3$-convex functions $f:[0,\infty
)\rightarrow\mathbb{R}$ such that $f(0)=f^{\prime}(0)=0$ are%
\[
x^{\alpha}\text{ }(\alpha\geq2),\text{ }e^{x}-1-x,\text{ }-x\log(x+1)\text{
and }-\log(\cosh x).
\]

More results concerning the $3$-convex/$3$-concave functions are made
available by the recent survey of Marinescu and Niculescu \cite{MN2023}.

\section{Some consequences of the Hardy-Littlewood-Pólya theorem of
majorization}

We start by recalling the Hardy-Littlewood-Pólya theorem of majorization:

\begin{theorem}
\label{thmHLP}Let $f$ be a real-valued convex function defined on a nonempty
interval $I.$ If $\mathbf{x}=(x_{k})_{k=1}^{n}$ and $\mathbf{y}=(y_{k}%
)_{k=1}^{n}$ are two families of points in $I$ such that%
\begin{equation}
x_{1}\geq\cdots\geq x_{n} \label{HLP1}%
\end{equation}
and%
\begin{equation}
\sum_{k=1}^{m}x_{k}\leq\sum_{k=1}^{m}y_{k}\quad\text{for }m=1,\dots,n,
\label{HLP2}%
\end{equation}

with equality for $m=n,$ then%
\begin{equation}
\sum_{k=1}^{n}f(x_{k})\leq\sum_{k=1}^{n}f(y_{k}). \label{HLP3}%
\end{equation}

When condition \emph{(\ref{HLP1})} is replaced by $y_{1}\leq\cdots\leq y_{n},$
then the conclusion \emph{(\ref{HLP3})} works in the reverse direction.
\end{theorem}

Pólya \cite{Pol1950} noticed that this result implies the other classical
result of majorization theory, precisely, the theorem of Tomi\'{c}
\cite{T1949} and Weyl \cite{W1949} on weak majorization:

\begin{corollary}
\emph{(The Tomi\'{c}-Weyl theorem)\label{cor_TW}} Let $f$ be a real-valued
nondecreasing convex function defined on a nonempty interval $I.$ Then the
inequality \emph{(\ref{HLP3})} holds for every two families $\mathbf{x}%
=(x_{k})_{k=1}^{n}$ and $\mathbf{y}=(y_{k})_{k=1}^{n}$ of points in $I$ that
verify the conditions \emph{(\ref{HLP1}) }and\emph{ (\ref{HLP2}).}

When condition \emph{(\ref{HLP1})} is replaced by $y_{1}\leq\cdots\leq y_{n},$
then the conclusion \emph{(\ref{HLP3})} works in the reverse direction.
\end{corollary}

The details can be found in \cite{NP2024}.

Of special interest for us will be the following\ "truncated forms" of
Corollary \ref{cor_TW} and Theorem \ref{thmHLP}.

\begin{theorem}
\label{thm_convex_trunc_maj}\emph{(Truncated majorization: the convex case)
}If $f:[0,\infty)\rightarrow\mathbb{R}$ is a nondecreasing convex function,
then for all integer numbers $2\leq m\leq n$ and all strings $x_{1}%
,\ldots,x_{n}$ and $y_{1},\ldots,y_{m}$ of positive numbers such that%
\begin{equation}
\max\left\{  \sum\limits_{p=1}^{k}x_{i_{p}}:i\,_{r}\neq i_{s}\right\}
\leq\max\left\{  \sum\limits_{p=1}^{k}y_{j_{p}}:j\,_{r}\neq j_{s}\right\}
\label{TW1}%
\end{equation}
for $k=1,...,m-1$ and
\begin{equation}
\sum_{k=1}^{n}x_{k}\leq\sum_{k=1}^{m}y_{k}, \label{TW2}%
\end{equation}
we have%
\begin{equation}
\sum_{k=1}^{n}f(x_{k})\leq\sum_{k=1}^{m}f(y_{k})+(n-m)f(0).
\label{ineq_HLP1_truncated}%
\end{equation}

\end{theorem}

\begin{proof}
Apply Corollary \ref{cor_TW} to the families
\[
\mathbf{x}=(x_{1},x_{2},...,x_{n})\text{\quad and\quad}\mathbf{y}%
=(y_{1},...,y_{m},\underset{n-m\text{ times}}{\underbrace{0,...,0}}).
\]

\end{proof}

\begin{corollary}
\label{cor_HLP1}If $f:[0,\infty)\rightarrow\mathbb{R}$ is a nondecreasing
convex function and $x_{1},x_{2},\ldots,\allowbreak x_{n},y_{1},y_{2}$
$(n\geq2)$ are nonnegative numbers such that
\[
\max\{x_{1},x_{2},...,x_{n}\}\leq\max\{y_{1},y_{2}\}
\]
and
\[
\sum\limits_{k=1}^{n}x_{k}\leq y_{1}+y_{2},
\]
then%
\[
\sum\limits_{k=1}^{n}f(x_{k})\leq f(y_{1})+f(y_{2})+(n-2)f(0).
\]
In the case of nonincreasing concave functions defined on $[0,\infty),$ the
conclusion works in the reverse direction.
\end{corollary}

\begin{theorem}
\label{thm_concave_trunc_maj}\emph{(Truncated majorization: the concave case)}
Let $f:[0,\infty)\rightarrow\mathbb{R}$ be an nondecreasing concave function
and let $x_{1},x_{2},\ldots,x_{n},y_{1},y_{2},...,y_{m}$ $(2\leq m\leq n)$ be
nonnegative numbers such that
\begin{align*}
\max\{x_{1},x_{2},...,x_{n}\}  &  \leq\max\{y_{1},y_{2},...,y_{m}\},\\
\max\left\{  x_{i_{1}}+x_{i_{2}}:i\,_{1}\neq i_{2}\right\}   &  \leq
\max\{y_{j_{1}}+y_{j_{2}}:j_{1}\neq j_{2}\}\\
&  ...\\
\max\left\{  \sum\limits_{p=1}^{k}x_{i_{p}}:i\,_{r}\neq i_{s}\right\}   &
\leq\max\left\{  \sum\limits_{p=1}^{k}y_{j_{p}}:j\,_{r}\neq j_{s}\right\}
\end{align*}
for $k\leq m-1$ and
\[
\sum\limits_{i=1}^{n}x_{i}\geq\sum\limits_{j=1}^{m}y_{j}.
\]
\qquad Then%
\begin{equation}
\sum\limits_{i=1}^{n}f(x_{i})\geq\sum\limits_{j=1}^{m}f(y_{j})+(n-m)f(0).
\label{ineq_HLP2_truncated}%
\end{equation}

\end{theorem}

\begin{proof}
Without loss of generality we may increase $y_{k}$ to $\tilde{y}_{k}$ (for
$k\in\{1,...,m\})$ so that
\[
\sum\limits_{k=1}^{n}x_{k}=\sum\limits_{j=1}^{m}\tilde{y}_{j}%
\]
and also we may assume that $x_{1}\geq x_{2}\geq x_{3}\geq\cdots\geq x_{n}$
and $\tilde{y}_{1}\geq\tilde{y}_{2}\geq\cdots\geq\tilde{y}_{m}$. Due to our
hypotheses $\tilde{y}_{1}\geq x_{1},$ $\tilde{y}_{1}+\tilde{y}_{2}\geq
x_{1}+x_{2},...,$ and $\tilde{y}_{1}+\tilde{y}_{2}+\cdots+\tilde{y}_{m}%
=\sum\limits_{k=1}^{n}x_{k}\geq x_{1}+\cdots+x_{m}.$ Therefore the string
$\mathbf{x}=(x_{1},x_{2},x_{3},...,x_{n})$ is majorized by $\mathbf{y}%
=(\tilde{y}_{1}+\tilde{y}_{2}+\cdots+\tilde{y}_{m},\underset{n-m\text{ times}%
}{\underbrace{0,...,0}})$ and the inequality (\ref{ineq_HLP2_truncated})
follows from the Hardy-Littlewood-Pólya theorem of majorization and the fact
that $f$ is nondecreasing.
\end{proof}

\begin{corollary}
\label{corHLP2}If $f:[0,\infty)\rightarrow\mathbb{R}$ is a nondecreasing
concave function and $x_{1},x_{2},\ldots,\allowbreak x_{n},y_{1},y_{2}$
$(n\geq2)$ are nonnegative numbers such that
\[
\max\{x_{1},x_{2},...,x_{n}\}\leq\max\{y_{1},y_{2}\}
\]
and
\[
\sum\limits_{k=1}^{n}x_{k}\geq y_{1}+y_{2},
\]
then%
\[
\sum\limits_{k=1}^{n}f(x_{k})\geq f(y_{1})+f(y_{2})+(n-2)f(0).
\]
In the case of nonincreasing convex functions defined on $[0,\infty),$ the
conclusion works in the reverse direction.
\end{corollary}

Theorem \ref{thm_concave_trunc_maj} and its Corollary \ref{corHLP2} apply to
functions such as $x^{\alpha}$ ($\alpha\in(0,1]),$ $\log(1+x),~\frac{x}{x+r}$
$(r\geq0)$ and $1-e^{-tx}$ $(t>0),$ all defined on $[0,\infty)$ and vanishing
at the origin.

A particular case of Corollary \ref{corHLP2} was noticed by Schötz \cite{Sch},
Lemma 47, using a different argument.

A nice illustration of Corollary \ref{cor_HLP1} (applied to the function
$f(x)=x^{p/q})$ is the following inequality noticed by Enflo \cite{Enflo1970}:
if $r_{1},r_{2},...,r_{6}$ are positive real numbers and $\max(r_{1}%
,\ldots,r_{4})\leq$ $\max(r_{5},r_{6})$ and $r_{1}^{q}+\cdots+r_{4}^{q}%
=r_{5}^{q}+r_{6}^{q},$ then%
\[
r_{1}^{p}+\cdots+r_{4}^{p}\leq r_{5}^{p}+r_{6}^{p}\quad\text{if }p\geq q.
\]
For applications to the functional inequalities see the next section.

\section{Functional inequalities via majorization}

The following inequality can be found in the book of Zhang \cite{Zhang},
Problem 36, p. 215: if $A,B,C$ are positive semidefinite matrices of the same
dimension, then%
\[
\det(A+B+C)+\det C\geq\det(A+C)+\det(B+C).
\]
Since $A+B+C\geq A+C,~B+C$ and the function $\det$ is nondecreasing on the
cone of positive semidefinite matrices, one can derive via Corollary
\ref{cor_HLP1} the following result.

\begin{theorem}
\label{thm_Zhang}If $f:[0,\infty)\rightarrow\mathbb{R}$ is a nondecreasing
convex function and $A,B$ and $C$ are positive semidefinite matrices then%
\[
f\left(  \det(A+B+C)\right)  +f\left(  \det C\right)  \geq f\left(
\det(A+C)\right)  +f\left(  \det(B+C)\right)  ,
\]
which \emph{(}by symmetrization\emph{)} leads to the following inequality,
\begin{multline*}
\frac{f\left(  \det A\right)  +f\left(  \det B\right)  +f\left(  \det
C\right)  }{3}+f\left(  \det(A+B+C)\right) \\
\geq\frac{2}{3}f\left(  \det(A+B)\right)  +f\left(  \det(B+C)\right)
+f\left(  \det(A+C)\right)  ,
\end{multline*}
that reminds us of Popoviciu's inequality \emph{\cite{Pop1965}}.
\end{theorem}

The case of the positive semidefinite matrices
\[
A=\left(
\begin{array}
[c]{cc}%
1 & 0\\
0 & 1
\end{array}
\right)  ,\quad B=\left(
\begin{array}
[c]{cc}%
1 & 0\\
0 & 2
\end{array}
\right)  ,\quad C=\left(
\begin{array}
[c]{cc}%
1 & 1\\
1 & 2
\end{array}
\right)  \quad
\]
and of the convex function $f(x)=x^{2},$ $x\geq0,$ shows that the conclusion
of Theorem \ref{thm_Zhang} cannot be strengthened to
\begin{multline*}
\frac{f\left(  \det A\right)  +f\left(  \det B\right)  +f\left(  \det
C\right)  }{3}+f\left(  \det(\frac{A+B+C}{3})\right) \\
\geq\frac{2}{3}\left(  f\left(  \det(\frac{A+B}{2})\right)  +f\left(
\det(\frac{B+C}{2})\right)  +f\left(  \det(\frac{A+C}{2})\right)  \right)  ,
\end{multline*}
that is, it is not a generalization of Popoviciu's inequality.

A partial generalization of Popoviciu's inequality for functions of a vector
variable can be found in \cite{N2021}. See also \cite{BNP2010}.

The statement of\emph{ the optimal }$2$\emph{-uniform convexity inequality}
stated below can be found in the paper of Ball, Carlen and Lieb \cite{BCL1994}%
,\emph{ }Proposition\emph{ }3:

\begin{proposition}
\label{prop2-unifconv}If $p\in(1,2]$ and $x$ and $y$ belong to an $L^{p}$
space $($or to the Schatten space $S^{p}(H))$, then%
\[
\left\Vert x\right\Vert _{p}^{2}+\left\Vert y\right\Vert _{p}^{2}%
\geq2\left\Vert \frac{x+y}{2}\right\Vert _{p}^{2}+2(p-1)\left\Vert \frac
{x-y}{2}\right\Vert _{p}^{2}.
\]
For $p\in\lbrack2,\infty),$ the inequality is reversed.

Here $\left\Vert \cdot\right\Vert _{p}$ denotes either the $L^{p}$ norm or the
Schatten norm of index $p$.
\end{proposition}

Recall that the Schatten space of index $p\in\lbrack1,\infty),$ associated to
a Hilbert space $H,$ is the space $S^{p}(H)$ of all compact linear operators
$T:H\rightarrow H$ whose sequences $(s_{n}(T))_{n}$ of singular values belong
to $\ell^{p}.$ $S^{p}(H)$ is a Banach space with respect to the norm%
\[
\left\Vert T\right\Vert _{p}=\left(
{\displaystyle\sum\nolimits_{n}}
\left\vert s_{n}(T)\right\vert ^{p}\right)  ^{1/p}.
\]
See \cite{Simon} for a comprehensive presentation of the theory of these spaces.

As noticed that Pisier and Xu \cite{PX}, Theorem 5.3, the Proposition
\ref{prop2-unifconv} also works in the context of noncommutative $L^{p}$ spaces.

Corollary \ref{corHLP2} allows us to extend Proposition \ref{prop2-unifconv}
as follows:

\begin{theorem}
\label{thm-extension of 2-unif cv}\emph{(}The $2$-uniform convexity functional
inequality for $p\in(1,2])$ If $p\in(1,2]$ and $f:[0,\infty)\rightarrow
\mathbb{R}$ is a nondecreasing function such that $f(0)=0$ and $f(x^{1/2})$ is
convex, then%
\[
f\left(  \left\Vert x\right\Vert _{p}\right)  +f\left(  \left\Vert
y\right\Vert _{p}\right)  \geq2f\left(  \left\Vert \frac{x+y}{2}\right\Vert
_{p}\right)  +\lfloor2(p-1)\rfloor f\left(  \left\Vert \frac{x-y}%
{2}\right\Vert _{p}\right)  ,
\]
for all elements $x$ and $y$ belonging to an $L^{p}$ space, to a Schatten
space $S^{p}(H)$ or to a noncommutative $L^{p}$ space$.$

Here $\lfloor\cdot\rfloor$ denotes the floor function $(\lfloor x\rfloor$ $=$
the greatest integer less than or equal to $x).$
\end{theorem}

According to Remark \ref{rem_cv of f (sqrt x)} \emph{(b)}, the hypotheses of
Theorem \ref{thm-extension of 2-unif cv} are fulfilled by every nondecreasing
and differentiable function $f:[0,\infty)\rightarrow\mathbb{R}$ such that
$f(0)=0,$ $f^{\prime}$ is convex and $f^{\prime}(0)=0.$

\begin{proof}
Clearly%
\[
\left\Vert \frac{x\pm y}{2}\right\Vert _{p}^{2}\leq\left(  \frac{\left\Vert
x\right\Vert _{p}+\left\Vert y\right\Vert _{p}}{2}\right)  ^{2}\leq
\max\left\{  \left\Vert x\right\Vert _{p}^{2},\left\Vert y\right\Vert _{p}%
^{2}\right\}
\]
and%
\[
\left\Vert \frac{x+y}{2}\right\Vert _{p}^{2}+\left\Vert \frac{x-y}%
{2}\right\Vert _{p}^{2}\leq\left\Vert x\right\Vert _{p}^{2}+\left\Vert
y\right\Vert _{p}^{2},
\]
so the conclusion follows from Corollary \ref{cor_HLP1}.
\end{proof}

The case $p>2$ of Theorem \ref{thm-extension of 2-unif cv} makes the objective
of Theorem \ref{thm-extension of 2-unif cv p>=2} in Section 6.

The next consequence of Corollary \ref{corHLP2} makes use of the following
special case of the parallelogram law: for all real numbers $a$ and $b,$%
\[
\left\vert a\right\vert ^{2}+\left\vert b\right\vert ^{2}=2\left\vert
\frac{a-b}{2}\right\vert ^{2}+2\left\vert \frac{a+b}{2}\right\vert ^{2}.
\]

\begin{corollary}
\label{cor_Clarkson}Suppose that $f:[0,\infty)\rightarrow\lbrack0,\infty)$ is
a nondecreasing, convex, and $3$-concave function such that $f(0)=0$. Then for
all real numbers $a$ and $b,$%
\[
f\left(  |a|\right)  +f\left(  \left\vert b\right\vert \right)  \leq2f\left(
\left\vert \frac{a-b}{2}\right\vert \right)  +2f\left(  \left\vert \frac
{a+b}{2}\right\vert \right)  .
\]
In particular, for all $\alpha\in\lbrack1,2],$
\[
|a|^{\alpha}+\left\vert b\right\vert ^{\alpha}\leq2\left\vert \frac{a-b}%
{2}\right\vert ^{\alpha}+2\left\vert \frac{a+b}{2}\right\vert ^{\alpha}.
\]

\end{corollary}

See Corollary \ref{cor_alfa power} below for an extension to the framework of
Banach spaces.

\begin{proof}
According to Theorem \ref{thm-comp}, the function $g(x)=f(x^{1/2})$ is
concave. As a consequence, the result of Corollary \ref{cor_Clarkson} follows
from Corollary \ref{corHLP2}, when applied to the function $g$ and to the
elements%
\begin{align*}
y_{1}  &  =\left\vert a\right\vert ^{2},~y_{2}=\left\vert b\right\vert
^{2},\text{ }\\
x_{1}  &  =x_{2}=\left\vert \frac{a-b}{2}\right\vert ^{2}\text{ and }%
x_{3}=x_{4}=\left\vert \frac{a+b}{2}\right\vert ^{2}.
\end{align*}

\end{proof}

Hanner's inequalities for $L^{p}$ spaces with $p\in(1,2]$ assert that
\begin{equation}
(\Vert u\Vert_{p}+\Vert v\Vert_{p})^{p}+\bigl\vert\Vert u\Vert_{p}-\Vert
v\Vert_{p}\bigr\vert^{p}\leq\Vert u+v\Vert_{p}^{p}+\Vert u-v\Vert_{p}^{p},
\label{eqHanner1-2}%
\end{equation}
while for $p\in\lbrack2,\infty)$ these inequalities work in the reversed
direction. See \cite{NP2018}, p. 139 for details.

The argument of Corollary \ref{cor_Clarkson} can be easily adapted to obtain
the following generalization of Hanner's aforementioned inequalities$:$

\begin{theorem}
\label{thm_Hanner p<=2}$($The generalization of Hanner's inequalities for
$p\in(1,2])$ Let $f:[0,\infty)\rightarrow\mathbb{R}$ be a nondecreasing convex
function such that $f(0)=0$ and $f(x^{1/p})$ is concave for some $p\in(1,2].$
Then%
\[
f\left(  \Vert u\Vert_{p}+\Vert u\Vert_{p}\right)  +f\left(  \left\vert \Vert
u\Vert_{p}-\Vert u\Vert_{p}\right\vert \right)  \leq f\left(  \Vert
u+v\Vert_{p}\right)  +f\left(  \Vert u-v\Vert_{p}\right)
\]
for all $u,v$ belonging to a Lebesgue space $L^{p}(\mu).$
\end{theorem}

The companion of Theorem \ref{thm_Hanner p<=2} for $p\geq2$ is a direct
consequence of Corollary \ref{cor_HLP1}:

\begin{theorem}
\label{thm_Hanner p>=2}$($The generalization of Hanner's inequalities for
$p\in\lbrack2,\infty))$ Let $f:[0,\infty)\rightarrow\mathbb{R}$ be a
nondecreasing function such that $f(0)=0$ and $f(x^{1/p})$ is concave. Then%
\[
f\left(  \Vert u\Vert_{p}+\Vert u\Vert_{p}\right)  +f\left(  \left\vert \Vert
u\Vert_{p}-\Vert u\Vert_{p}\right\vert \right)  \geq f\left(  \Vert
u+v\Vert_{p}\right)  +f\left(  \Vert u-v\Vert_{p}\right)
\]
for all $u,v$ belonging to a Lebesgue space $L^{p}(\mu).$
\end{theorem}

As noticed by Ball, Carlen and Lieb \cite{BCL1994}, Theorem 2, the
inequalities (\ref{eqHanner1-2}) also hold in the context of Schatten spaces
$S^{p}(H)$ in the following two cases:

\begin{enumerate}
\item[(HS1)] $1<p\leq4/3$; and

\item[(HS2)] $u$ and $v$ belong to $S^{p}(H)$ for some $p\in(1,2]$ and $u\pm
v$ are positive semidefinite.
\end{enumerate}

For $p\geq2$, the inequalities (\ref{eqHanner1-2}) work in the reverse
direction and the restriction in (HS1) becomes $p\geq4$, and the restriction
in (HS2) changes to the restriction that $u$ and $v$ are positive
semidefinite. Subject to these restrictions, Theorem \ref{thm_Hanner p<=2} and
Theorem \ref{thm_Hanner p>=2} continue to work in the context of Schatten spaces.

\section{A substitute of the parallelogram law in the context of Banach
spaces}

In connection with the famous work \cite{JN} of Jordan and von Neumann
concerning the inner product spaces, Clarkson \cite{Cl1937} has introduced the
\emph{von Neumann-Jordan constant} $C_{NJ}(X)$ of a Banach space $X$ as
\[
C_{NJ}(X)=\sup\left\{  \frac{\left\Vert u+v\right\Vert ^{2}+\left\Vert
u-v\right\Vert ^{2}}{2\left\Vert u\right\Vert ^{2}+2\left\Vert v\right\Vert
^{2}}:u,v\in X\text{ and }\left\Vert u\right\Vert +\left\Vert v\right\Vert
\neq0\right\}  .
\]

We have $1\leq C_{NJ}(X)\leq2$ for all Banach spaces $X$ and $C_{NJ}(X)=1$
\emph{if and only if} $X$ \emph{is a Hilbert space}.

In general, $C_{NJ}(X)<2$ for any uniformly convex space. $C_{NJ}(X)=2$ in the
case of spaces $L^{p}(\mathbb{R})$ with $p=1$\ or $p=\infty$ and the same is
true in the case of Banach spaces of continuous functions endowed with the sup norm.

Let $1<p<\infty$ and $t=\min\left\{  p,p/(p-1)\right\}  .$ Then
\[
C_{NJ}(X)=2^{2/t-1}%
\]
for each of the following Banach spaces $X$ of dimension at least 2:

\begin{enumerate}
\item[-] $L^{p}(\mathbb{R})$ (see Clarkson \cite{Cl1937});

\item[-] the Sobolev spaces $X=W^{k,p}(\mathbb{R})$\ (see Kato and Miyazaki
\cite{KM96});

\item[-] the\emph{ }Schatten classes of index $p,$ $S_{p}(H)$ (Kato and
Takahashi \cite{KT98}).
\end{enumerate}

The definition of the von Neumann-Jordan constant gives rise to the following
generalization of the parallelogram rule:%
\begin{equation}
\left\Vert u\right\Vert ^{2}+\left\Vert v\right\Vert ^{2}\leq2C_{NJ}%
(X)\left\Vert \frac{u-v}{2}\right\Vert ^{2}+2C_{NJ}(X)\left\Vert \frac{u+v}%
{2}\right\Vert ^{2} \label{genParLaw}%
\end{equation}
for all $u,v\in X.$

The next result extends Theorem 1 to the general context of Banach spaces. Its
statement makes use of a modification of the von Neumann-Jordan constant,
precisely,%
\[
N(X)=\left\{
\begin{array}
[c]{cl}%
2C_{NJ}(X) & \text{if }2C_{NJ}(X)\text{ is an integer}\\
4 & \text{otherwise,}%
\end{array}
\right.  ~
\]
motivated by the use of Corollary \ref{corHLP2} which deals with strings of
elements indexed by integers. Notice that $N(X)=3$ in the case of $L^{p}$
spaces with $p=\left(  2\log2\right)  /\left(  \log3\right)  \approx
1.\,\allowbreak261\,859\,507\ldots$ .

\begin{theorem}
\label{thm_functional-Par-law}Let $X$ be a Banach space and let $f:[0,\infty
)\rightarrow\mathbb{R}$ be a nondecreasing, convex and $3$-concave function
such that $f(0)=0.$ Then%
\begin{equation}
f\left(  \left\Vert u\right\Vert \right)  +f\left(  \left\Vert v\right\Vert
\right)  \leq N(X)f\left(  \left\Vert \frac{u-v}{2}\right\Vert \right)
+N(X)f\left(  \left\Vert \frac{u+v}{2}\right\Vert \right)  ,
\label{functional-Par-Law}%
\end{equation}
and%
\begin{multline*}
f\left(  \left\Vert u\right\Vert \right)  +f\left(  \left\Vert v\right\Vert
\right)  \leq\frac{N(X)}{2}\left\{  f\left(  \left\Vert u+x\right\Vert
\right)  +f\left(  \left\Vert v+x\right\Vert \right)  \right. \\
\left.  +f\left(  \left\Vert x\right\Vert \right)  +f\left(  \left\Vert
u+v+x\right\Vert \right)  \right\}  ,
\end{multline*}
for all \thinspace$u,v,x\in X.$
\end{theorem}

\begin{proof}
We will apply Corollary \ref{corHLP2} to the function $f\circ\sqrt{\cdot}$ and
the points%
\begin{gather*}
x_{1}=\cdots=x_{N(X)}=\left\Vert \frac{u-v}{2}\right\Vert ^{2},\text{ }\\
x_{N(X)+1}=\cdots=x_{2N(X)}=\left\Vert \frac{u+v}{2}\right\Vert ^{2},\\
y_{1}=\left\Vert u\right\Vert ^{2}\text{ and }y_{2}=\left\Vert v\right\Vert
^{2}.
\end{gather*}
We have $\max\{x_{1},x_{2},\ldots,x_{2N(X)}\}\leq\max\{y_{1},y_{2}\}$ since
\begin{align*}
x_{1}  &  =\cdots=x_{N(X)}=\left\Vert \frac{u-v}{2}\right\Vert ^{2}\leq\left(
\frac{\left\Vert u\right\Vert +\left\Vert v\right\Vert }{2}\right)  ^{2}\\
&  \leq\max\left\{  \left\Vert u\right\Vert ^{2},\left\Vert v\right\Vert
^{2}\right\}  =\max\{y_{1},y_{2}\}
\end{align*}
and the same is true for $x_{N(X)+1},\ldots,x_{2N(X)}.~$The fact that
$x_{1}+x_{2}+x_{3}+\ldots+x_{2N(X)}\geq y_{1}+y_{2}$ follows from the
inequality (\ref{genParLaw}), while Theorem \ref{thm-comp} assures that
$f\circ\sqrt{\cdot}$ is a nondecreasing concave function. The inequality
(\ref{functional-Par-Law}) is now clear.

The second inequality in the statement of Theorem \ref{thm_functional-Par-law}
is a consequence of the inequality (\ref{functional-Par-Law}). Indeed, since
$f$ is convex and nondecreasing, the function $f\circ\left\Vert \cdot
\right\Vert $ is also convex, which yields%
\begin{align*}
2f\left(  \left\Vert \frac{u-v}{2}\right\Vert \right)   &  \leq f\left(
\left\Vert u+x\right\Vert \right)  +f\left(  \left\Vert -v-x\right\Vert
\right) \\
&  =f\left(  \left\Vert u+x\right\Vert \right)  +f\left(  \left\Vert
v+x\right\Vert \right)  ,
\end{align*}
and
\[
2f\left(  \left\Vert \frac{u+v}{2}\right\Vert \right)  \leq f\left(
\left\Vert x\right\Vert \right)  +f\left(  \left\Vert u+v+x\right\Vert
\right)  ,
\]
for all $u,v,x\in\mathbb{R}^{N}.$
\end{proof}

In the particular case when $f$ is the function $x^{\alpha}$ with $\alpha
\in\lbrack1,2],$ Theorem \ref{thm_functional-Par-law} yields the following result:

\begin{corollary}
\label{cor_alfa power}Let $X$ be a Banach and $\alpha\in\lbrack1,2].$ Then
\[
\left\Vert u\right\Vert ^{\alpha}+\left\Vert v\right\Vert ^{\alpha}\leq
N(X)\left(  \left\Vert \frac{u-v}{2}\right\Vert ^{\alpha}+\left\Vert
\frac{u+v}{2}\right\Vert ^{\alpha}\right)  \leq N(X)^{2}\left(  \left\Vert
u\right\Vert ^{\alpha}+\left\Vert v\right\Vert ^{\alpha}\right)
\]
whenever $u,v\in X.$
\end{corollary}

\section{The generalization of the quadruple inequality of Schötz}

We are now in a position to state the following generalization to the context
of Banach spaces of the quadruple inequality of Schötz (see Theorem 1 in the
Introduction) :

\begin{theorem}
\label{thm_Schotz_gen}Let $y,z,q,r$ be four points in the Banach space $X$ and
let $f$ be a nondecreasing, convex and $3$-concave function defined on
$[0,\infty)$ such that $f(0)=0.$ Then%
\begin{multline*}
f\left(  \left\Vert y-q\right\Vert \right)  +f\left(  \left\Vert
z-r\right\Vert \right) \\
\leq\frac{N(X)}{2}\left\{  f\left(  \left\Vert y-z\right\Vert \right)
+f\left(  \left\Vert r-q\right\Vert \right)  \right.  \left.  +f\left(
\left\Vert z-q\right\Vert \right)  +f\left(  \left\Vert y-r\right\Vert
\right)  \right\}  .
\end{multline*}

\end{theorem}

In the case of inner product spaces, $N(X)=2$ and we retrieve the quadruple
inequality of Schötz.

\begin{proof}
Given four points $y,z,q,r$ in the space $X,$ let us denote%
\[
z-q=x,\text{ }q-y=u,\ \text{and}\ r-z=v.
\]
Then%
\[
z-y=u+x,\text{ }r-y=u+v+x,\ \text{and}\ r-q=v+x,
\]
so the proof ends by taking into account the second part of Theorem
\ref{thm_functional-Par-law}.\quad$\square$
\end{proof}

In the case of $L^{p}$ spaces, Clarkson (see \cite{Cl1936}, Theorem 2) noticed
the following two inequalities, usually known as \emph{the easy Clarkson
inequalities}:
\begin{gather*}
2^{p-1}\left(  \left\Vert x\right\Vert _{p}^{p}+\left\Vert y\right\Vert
_{p}^{p}\right)  \leq\left\Vert x-y\right\Vert _{p}^{p}+\left\Vert
x+y\right\Vert _{p}^{p}\leq2\left(  \left\Vert x\right\Vert _{p}%
^{p}+\left\Vert y\right\Vert _{p}^{p}\right)  ,\quad\text{if }p\in(1,2],\\
2\left(  \left\Vert x\right\Vert _{p}^{p}+\left\Vert y\right\Vert _{p}%
^{p}\right)  \leq\left\Vert x-y\right\Vert _{p}^{p}+\left\Vert x+y\right\Vert
_{p}^{p}\leq2^{p-1}\left(  \left\Vert x\right\Vert _{p}^{p}+\left\Vert
y\right\Vert _{p}^{p}\right)  ,\quad\text{if }p\in\lbrack2,\infty).
\end{gather*}
As a consequence, by replacing the inequalities (\ref{genParLaw}) with the
easy Clarkson inequalities and using a similar argument to the one that we
used for Theorem \ref{thm_functional-Par-law} and Theorem \ref{thm_Schotz_gen}%
, we arrive at the following companion of these theorems.

\begin{theorem}
\label{thmLpquadruple}Consider an $L^{p}$ space $X$ with $p\in(1,2]$ and let
$f:[0,\infty)\rightarrow\mathbb{R}$ be a nondecreasing, convex function such
that $f(0)=0$ and $f(x^{1/p})$ is concave. Then%
\[
f\left(  \left\Vert u\right\Vert _{p}\right)  +f\left(  \left\Vert
v\right\Vert _{p}\right)  \leq2f\left(  \left\Vert \frac{u-v}{2}\right\Vert
_{p}\right)  +2f\left(  \left\Vert \frac{u+v}{2}\right\Vert _{p}\right)  ,
\]
and%
\[
f\left(  \left\Vert u\right\Vert _{p}\right)  +f\left(  \left\Vert
v\right\Vert _{p}\right)  \leq f\left(  \left\Vert u+x\right\Vert _{p}\right)
+f\left(  \left\Vert v+x\right\Vert _{p}\right)  +f\left(  \left\Vert
x\right\Vert _{p}\right)  +f\left(  \left\Vert u+v+x\right\Vert _{p}\right)
,
\]
for all \thinspace$u,v,x\in X.$ As a consequence,%
\begin{multline*}
f\left(  \left\Vert y-q\right\Vert _{p}\right)  +f\left(  \left\Vert
z-r\right\Vert _{p}\right) \\
\leq f\left(  \left\Vert y-z\right\Vert _{p}\right)  +f\left(  \left\Vert
r-q\right\Vert _{p}\right)  +f\left(  \left\Vert z-q\right\Vert _{p}\right)
+f\left(  \left\Vert y-r\right\Vert _{p}\right)
\end{multline*}
for all \thinspace$q,r,y,z\in X.$

For $p\in\lbrack2,\infty)$, these inequalities should be replaced by the
following ones:%
\begin{multline*}
f\left(  \left\Vert u\right\Vert _{p}\right)  +f\left(  \left\Vert
v\right\Vert _{p}\right)  \leq C(p)f\left(  \left\Vert \frac{u-v}%
{2}\right\Vert _{p}\right)  +C(p)f\left(  \left\Vert \frac{u+v}{2}\right\Vert
_{p}\right)  ,\\
f\left(  \left\Vert u\right\Vert _{p}\right)  +f\left(  \left\Vert
v\right\Vert _{p}\right)  \leq\frac{C(p)}{2}\left\{  f\left(  \left\Vert
u+x\right\Vert _{p}\right)  +f\left(  \left\Vert v+x\right\Vert _{p}\right)
\right. \\
\left.  +f\left(  \left\Vert x\right\Vert _{p}\right)  +f\left(  \left\Vert
u+v+x\right\Vert _{p}\right)  \right\}  ,
\end{multline*}
and%
\begin{multline*}
f\left(  \left\Vert y-q\right\Vert _{p}\right)  +f\left(  \left\Vert
z-r\right\Vert _{p}\right)  \leq\frac{C(p)}{2}\left\{  f\left(  \left\Vert
y-z\right\Vert _{p}\right)  +f\left(  \left\Vert r-q\right\Vert _{p}\right)
\right. \\
\left.  +f\left(  \left\Vert z-q\right\Vert _{p}\right)  +f\left(  \left\Vert
y-r\right\Vert _{p}\right)  \right\}  ,
\end{multline*}
where $C(p)=2^{p-1}$ if $2^{p-1}$ is an integer and $C(p)=\lfloor
2^{p-1}+1\rfloor$ otherwise.
\end{theorem}

As McCarthy noticed in \cite{McC1967} (see also Simon \cite{Simon}), the two
easy Clarkson inequalities also work in the context of Schatten classes
$S^{p}(H)$ (provided that the $L^{p}$ norms $\left\Vert \cdot\right\Vert _{p}$
are replaced by the Schatten $p$ norms $\left\Vert \cdot\right\Vert _{p}$).
Accordingly, Theorem \ref{thmLpquadruple} still works in the framework of
Schatten classes.

We end this section by noticing a quadruple inequality (similar to that stated
in Theorem \ref{thmLpquadruple}) that results in the case $p\in\lbrack
2,\infty)$ of the optimal 2-uniform convexity inequality. For convenience, we
denote%
\[
\tilde{C}\left(  p\right)  =\left\{
\begin{array}
[c]{ll}%
\left(  p-1\right)  /2 & \text{if }\left(  p-1\right)  /2\in\mathbb{Z}\\
\lfloor\left(  p-1\right)  /2+1\rfloor & \text{otherwise.}%
\end{array}
\right.
\]

\begin{theorem}
\label{thm-extension of 2-unif cv p>=2}$($The $2$-uniform convexity functional
inequality: case $p\in\lbrack2,\infty))$ If $p\in\lbrack2,\infty)$ and
$f:[0,\infty)\rightarrow\mathbb{R}$ is a nondecreasing, convex, and
$3$-concave function such that $f(0)=0$. Then%
\[
f\left(  \left\Vert x\right\Vert _{p}\right)  +f\left(  \left\Vert
y\right\Vert _{p}\right)  \leq2f\left(  \left\Vert \frac{x+y}{2}\right\Vert
_{p}\right)  +2\tilde{C}\left(  p\right)  f\left(  \left\Vert \frac{x-y}%
{2}\right\Vert _{p}\right)
\]
and%
\begin{align*}
f\left(  \left\Vert x\right\Vert _{p}\right)  +f\left(  \left\Vert
y\right\Vert _{p}\right)   &  \leq f\left(  \left\Vert v\right\Vert \right)
+f\left(  \left\Vert x+y+v\right\Vert _{p}\right) \\
&  +\tilde{C}\left(  p\right)  f\left(  \left\Vert x+u\right\Vert _{p}\right)
+\tilde{C}\left(  p\right)  f\left(  \left\Vert y+u\right\Vert _{p}\right)
\end{align*}
for all elements $x,y,$ and $u$ belonging to an $L^{p}$ space, to a Schatten
space $S^{p}(H)$ or to a non-commutative $L^{p}$ space.

As a consequence,%
\begin{multline*}
f\left(  \left\Vert y-q\right\Vert _{p}\right)  +f\left(  \left\Vert
z-r\right\Vert _{p}\right)  \leq f\left(  \left\Vert z-q\right\Vert
_{p}\right)  +f\left(  \left\Vert y-r\right\Vert _{p}\right) \\
+\tilde{C}\left(  p\right)  f\left(  \left\Vert y-z\right\Vert _{p}\right)
+\tilde{C}\left(  p\right)  f\left(  \left\Vert r-q\right\Vert _{p}\right)
\end{multline*}
for all \thinspace$q,r,y,z$ in an $L^{p}$ space $($or in a Schatten space
$S^{p}(H)).$
\end{theorem}

\section{Further comments}

Not everything can be obtained via majorization theory.

Here is an example based on Fréchet's identity (see \cite{Fr1935}),%
\begin{equation}
\left\Vert \mathbf{x}\right\Vert ^{2}+\left\Vert \mathbf{y}\right\Vert
^{2}+\left\Vert \mathbf{z}\right\Vert ^{2}+\left\Vert \mathbf{x}%
+\mathbf{y}+\mathbf{z}\right\Vert ^{2}=\left\Vert \mathbf{x}+\mathbf{y}%
\right\Vert ^{2}+\left\Vert \mathbf{y}+\mathbf{z}\right\Vert ^{2}+\left\Vert
\mathbf{z}+\mathbf{x}\right\Vert ^{2}, \tag{$Fr$}\label{Fr}%
\end{equation}
valid in any Euclidean space $\mathbb{R}^{N}$. As Ressel noticed in
\cite{Res}, Theorem 2, this identity implies the Hornich-Hlawka inequality in
$\mathbb{R}^{N}$,
\begin{equation}
\left\Vert \mathbf{x}\right\Vert +\left\Vert \mathbf{y}\right\Vert +\left\Vert
\mathbf{z}\right\Vert +\left\Vert \mathbf{x}+\mathbf{y}+\mathbf{z}\right\Vert
\geq\left\Vert \mathbf{x}+\mathbf{y}\right\Vert +\left\Vert \mathbf{y}%
+\mathbf{z}\right\Vert +\left\Vert \mathbf{z}+\mathbf{x}\right\Vert ,
\tag{$HH$}\label{HH}%
\end{equation}
which in turn yields
\begin{multline*}
\left\Vert \mathbf{x}\right\Vert ^{1/2^{n}}+\left\Vert \mathbf{y}\right\Vert
^{1/2^{n}}+\left\Vert \mathbf{z}\right\Vert ^{1/2^{n}}+\left\Vert
\mathbf{x}+\mathbf{y}+\mathbf{z}\right\Vert ^{1/2^{n}}\\
\geq\left\Vert \mathbf{x}+\mathbf{y}\right\Vert ^{1/2^{n}}+\left\Vert
\mathbf{y}+\mathbf{z}\right\Vert ^{1/2^{n}}+\left\Vert \mathbf{z}%
+\mathbf{x}\right\Vert ^{1/2^{n}}%
\end{multline*}
for all $\mathbf{x},\mathbf{y},\mathbf{z}\in\mathbb{R}^{N}$ and all integers
$n=1,2,3,...$ .

One may think of the Hornich-Hlawka inequality as a particular case of a
functional inequality attached to Fréchet's identity and including the square
root function in its domain. So far, no such functional inequality is known,
despite a great deal of attention received by the Hornich-Hlawka inequality
over the years. See \cite{MN2024} for details.

To understand why the majorization theorem cannot be helpful for proving the
implication $(Fr)$ $\Longrightarrow(HH),$ let's restrict ourselves to the
positive cone $\mathbb{R}_{+}^{N}.$ In this case one can add to Fréchet's
identity the inequalities%
\begin{align*}
\max\left\{  \left\Vert \mathbf{x}\right\Vert ^{2},\left\Vert \mathbf{y}%
\right\Vert ^{2},\left\Vert \mathbf{z}\right\Vert ^{2},\left\Vert
\mathbf{x}+\mathbf{y}+\mathbf{z}\right\Vert ^{2}\right\}   &  \geq\max\left\{
\left\Vert \mathbf{x}+\mathbf{y}\right\Vert ^{2},\left\Vert \mathbf{y}%
+\mathbf{z}\right\Vert ^{2},\left\Vert \mathbf{z}+\mathbf{x}\right\Vert
^{2}\right\}  ,\\
\left\Vert \mathbf{z}\right\Vert ^{2}+\left\Vert \mathbf{x}+\mathbf{y}%
+\mathbf{z}\right\Vert ^{2}  &  \geq\left\Vert \mathbf{y}+\mathbf{z}%
\right\Vert ^{2}+\left\Vert \mathbf{z}+\mathbf{x}\right\Vert ^{2},
\end{align*}
and the similar ones obtained by permuting the variables. It becomes clear
that none of the strings
\[
X=(\left\Vert \mathbf{x}\right\Vert ^{2},\left\Vert \mathbf{y}\right\Vert
^{2},\left\Vert \mathbf{z}\right\Vert ^{2},\left\Vert \mathbf{x}%
+\mathbf{y}+\mathbf{z}\right\Vert ^{2})
\]
and
\[
Y=\left(  \left\Vert \mathbf{x}+\mathbf{y}\right\Vert ^{2},\left\Vert
\mathbf{y}+\mathbf{z}\right\Vert ^{2},\left\Vert \mathbf{z}+\mathbf{x}%
\right\Vert ^{2},0\right)
\]
is majorized by the other one in the sense of Hardy, Littlewood and Pólya.

However, some particular results may still be obtained.

For example, from Corollary \ref{corHLP2} one can deduce the following fact.

\begin{theorem}
\label{prop_Fr_functional}For every\ nondecreasing and concave function
$f:[0,\infty)\rightarrow\mathbb{R}$ and every $\mathbf{x},\mathbf{y}%
,\mathbf{z\in}\mathbb{R}_{+}^{N}.$
\begin{multline*}
f\left(  \left\Vert \mathbf{x}+\mathbf{y}\right\Vert ^{2}\right)  +f\left(
\left\Vert \mathbf{y}+\mathbf{z}\right\Vert ^{2}\right)  +f\left(  \left\Vert
\mathbf{z}+\mathbf{x}\right\Vert ^{2}\right) \\
\geq f\left(  \left\Vert \mathbf{x}\right\Vert ^{2}+\left\Vert \mathbf{y}%
\right\Vert ^{2}+\left\Vert \mathbf{z}\right\Vert ^{2}\right)  +f\left(
\left\Vert \mathbf{x}+\mathbf{y}+\mathbf{z}\right\Vert ^{2}\right)  +f(0).
\end{multline*}
The particular case where $f$ is the square root function gives rise to an
inequality that complements the Hornich-Hlawka inequality:
\begin{equation}
\left\Vert \mathbf{x}+\mathbf{y}\right\Vert +\left\Vert \mathbf{y}%
+\mathbf{z}\right\Vert +\left\Vert \mathbf{z}+\mathbf{x}\right\Vert
\geq\left(  \left\Vert \mathbf{x}\right\Vert ^{2}+\left\Vert \mathbf{y}%
\right\Vert ^{2}+\left\Vert \mathbf{z}\right\Vert ^{2}\right)  ^{1/2}%
+\left\Vert \mathbf{x}+\mathbf{y}+\mathbf{z}\right\Vert \label{RevHH}%
\end{equation}
for every $\mathbf{x},\mathbf{y},\mathbf{z\in}\mathbb{R}_{+}^{N}$.
\end{theorem}

Notice that the inequality (\ref{RevHH}) does not work for arbitrary elements
of $\mathbb{R}^{N}.$

The square norm is strongly superadditive on the positive cone of
$\mathbb{R}^{N},$ that is,
\[
\left\Vert \mathbf{x}+\mathbf{y}+\mathbf{z}\right\Vert ^{2}+\left\Vert
z\right\Vert ^{2}\geq\left\Vert \mathbf{x}+\mathbf{z}\right\Vert
^{2}+\left\Vert \mathbf{y}+\mathbf{z}\right\Vert ^{2}%
\]
for every $\mathbf{x},\mathbf{y},\mathbf{z\in}\mathbb{R}_{+}^{N}$. By
proceeding as in the case of Theorem \ref{thm_Zhang} one obtains the following
Popoviciu like inequality.

\begin{proposition}
\label{prop_Pop_functional}For every $\mathbf{x},\mathbf{y},\mathbf{z\in
}\mathbb{R}_{+}^{N}$ and every nondecreasing function $f:[0,\infty
)\rightarrow\mathbb{R}$ such that $f(0)=0$ and $f(x^{1/2})$ is convex, we have%
\begin{multline*}
\frac{f\left(  \left\Vert x\right\Vert \right)  +f\left(  \left\Vert
y\right\Vert \right)  +f\left(  \left\Vert z\right\Vert \right)  }{3}+f\left(
\left\Vert \mathbf{x}+\mathbf{y}+\mathbf{z}\right\Vert \right) \\
\geq\frac{2}{3}f\left(  \left\Vert \mathbf{x}+\mathbf{y}\right\Vert \right)
+f\left(  \left\Vert \mathbf{y}+\mathbf{z}\right\Vert \right)  +f\left(
\left\Vert \mathbf{z}+\mathbf{x}\right\Vert \right)  .
\end{multline*}

\end{proposition}

Similarly, starting from\ Serre's determinantal inequality\textbf{
}\cite{Serre},%
\begin{multline*}
\det\nolimits^{1/2}A+\det\nolimits^{1/2}B+\det\nolimits^{1/2}C+\det
\nolimits^{1/2}(A+B+C)\\
\leq\det\nolimits^{1/2}\left(  A+B\right)  +\det\nolimits^{1/2}\left(
B+C\right)  +\det\nolimits^{1/2}\left(  C+A\right)  ,
\end{multline*}
valid for all triplets $A,B,C$ of positive semidefinite matrices, one obtains
from Corollary \ref{corHLP2} the functional inequality%
\begin{multline*}
f\left(  \det\nolimits^{1/2}\left(  A+B\right)  \right)  +f\left(
\det\nolimits^{1/2}\left(  B+C\right)  \right)  +f\left(  \det\nolimits^{1/2}%
\left(  C+A\right)  \right) \\
\geq f\left(  \det\nolimits^{1/2}A+\det\nolimits^{1/2}B+\det\nolimits^{1/2}%
C\right)  +f\left(  \det\nolimits^{1/2}(A+B+C)\right)  ,
\end{multline*}
for every nonnegative,\ nondecreasing and concave function $f:[0,\infty
)\rightarrow\mathbb{R}.$

We leave open the problem whether Propositions \ref{prop_Fr_functional} and
\ref{prop_Pop_functional} still work for all triplets of vectors in
$\mathbb{R}^{N}.$

\medskip

\noindent\textbf{\noindent Acknowledgement. }The author would like to thank
\c{S}tefan Cobza\c{s}, Dan-\c{S}tefan Marinescu and Christof Schötz for many
useful comments on the subject of this paper.

\end{document}